\documentclass[a4paper]{article}

\usepackage[14pt]{extsizes}
\usepackage[left=25mm, top=25mm, right=15mm, bottom=25mm, nohead, footskip=10mm]{geometry}
\usepackage[english, ukrainian]{babel}
\usepackage{amsthm, amssymb, amsmath}
\usepackage{enumitem}
\usepackage{color}
\usepackage[square,comma,numbers]{natbib}
\usepackage{pifont}

\newtheorem{theorem}{Теорема}
\newtheorem{lemma}{Лема}
\newtheorem{corollary}{Наслідок}
\newtheorem{remark}{Зауваження}

\makeatletter
\newcommand{\leqnomode}{\tagsleft@true\let\veqno\@@leqno}
\makeatother

\newcommand\R{\mathbb R}
\newcommand\E{\mathbb E}
\newcommand\N{\mathbb N}
\newcommand\B{\mathcal B}
\newcommand\F{\mathcal F}
\newcommand\FF{\mathbb F}
\newcommand\I{\mathbb I}
\newcommand\PP{\mathcal P}
\newcommand\ltag[1]{\leqnomode\tag{#1}}
\renewcommand\d{\mathrm d}
\renewcommand\P{\mathbb P}

\setenumerate[1]{label={(\arabic*)}}

\title{\vspace{-2cm}\large{Про асимптотику розв'язків стохастичних диференціальних рівнянь зі стрибками \\ (On Asymptotics of Solutions of Stochastic Differential Equations with Jumps) \\ \small{УДК 519.21}}}
\author{В.~К. Юськович (V.~K. Yuskovych) \\ \small{вул.~Машинобудівна, 41, м.~Київ; viktyusk@gmail.com; НТУУ <<КПІ ім. І.~Сікорського>>}}
\date{}

\begin{document}
	\setlength{\belowdisplayskip}{5pt} \setlength{\belowdisplayshortskip}{2pt}
	\setlength{\abovedisplayskip}{5pt} \setlength{\abovedisplayshortskip}{2pt}
	
	\maketitle
	
	\begin{abstract}
		Розглянемо одновимірне стохастичне диференціальне рівняння зі стрибками
		$$\d X(t) = a(X(t))\d t + \sum_{k = 1}^m b_k(X(t-))\d Z_k(t),$$
		де $Z_k, \ k \in \{1, 2, ..., m\}$ -- незалежні центровані процеси Леві зі скінченними другими моментами. Ми доводимо, що якщо коефіцієнт $a(x)$ має певну степеневу асимптотику при $x \to \infty$, а коефіцієнти $b_k, \ k \in \{1, 2, ..., m\},$ задовольняють певну умову на зростання, то розв'язок $X(t)$ м.~н. має таку саму асимптотику при $t \to \infty$, що і розв'язок звичайного диференціального рівняння $\d x(t) = a(x(t))\d t$.
	\end{abstract}

	\selectlanguage{english}
	\begin{abstract}
		Consider a one-dimensional stochastic differential equation with jumps
		$$\d X(t) = a(X(t))\d t + \sum_{k = 1}^m b_k(X(t-))\d Z_k(t),$$
		where $Z_k, \ k \in \{1, 2, ..., m\}$ are independent centered L\'evy processes with finite second moments. We prove that if coefficient $a(x)$ has certain power asymptotics as $x \to \infty$ and coefficients $b_k, \ k \in \{1, 2, ..., m\},$ satisfy certain growth condition then a solution $X(t)$ has the same asymptotics as a solution of $\d x(t) = a(x(t))\d t$ as $t \to \infty$ a.s.
	\end{abstract}
	\selectlanguage{ukrainian}

	\section{Вступ}
	
	Як правило, розглядають два типи поведінки розв'язків стохастичних диференціальних рівнянь при $t \to \infty$: прямування до нескінченності та рекурентність. У даній статті ми будемо припускати, що розв'язок стохастичного диференціального рівняння прямує до нескінченності, та досліджувати його точну асимптотику.
	
	Уперше це питання розглядали Гіхман та Скороход \cite{gikhman1982stochastic} для одновимірного стохастичного диференціального рівняння вигляду
	\begin{equation}\label{sde}
		\d X(t) = a(X(t))\d t + b(X(t))\d W(t),
	\end{equation}
	де $W$ -- одновимірний вінерівський процес. Зокрема, вони знайшли достатні умови того, що $X(t) \to +\infty, \ t \to \infty$, та $X(t) \sim x(t), \ t \to \infty$, м.~н., де $x$ -- розв'язок звичайного диференціального рівняння
	\begin{equation}\label{ode}
		\d x(t) = a(x(t))\d t.
	\end{equation}
	Пізніше ця задача досліджувалася у роботі \cite{keller1984asymptotic}. У роботі \cite{buldygin2018asymptotic} розглядаються деякі типи неавтономних стохастичних диференціальних рівнянь. У статтях \cite{pavlyukevich2020generalized}, \cite{pilipenko2018perturbations} розглядаються стохастичні диференціальні рівняння з негауссівським шумом. 
	
	У книзі \cite{friedman2012stochastic} досліджується питання прямування до нескiнченностi та рекурентностi розв’язку системи лiнiйних стохастичних диференцiальних рiвнянь, а також поведiнки полярного кута розв’язку двовимiрного стохастичного диференціального рівняння. У статті \cite{samoilenko2012asymptotic} дослiджується асимптотична поведiнка багатовимiрних стохастичних диференцiальних рiвнянь шляхом порiвняння з лiнiйними звичайними диференцiальними рiвняннями.  У статті \cite{yuskovych2023asymptotic} розглядалося багатовимірне стохастичне диференціальне рівняння вигляду (\ref{sde}) та досліджувалася м.~н. поведінка розв'язку при $t \to \infty$: умови прямування модуля розв'язку до нескінченності, стабілізації кута $X(t)/|X(t)|$ та асимптотика модуля розв'язку. У статті \cite{pilipenko2018selection} аналогічне питання вивчалося для випадку адитивного шуму Леві.
	
	Питання про асимптотичну поведінку стохастичних диференціальних рівнянь з мультиплікативним шумом Леві не досліджене. У даній статті ми розглянемо стохастичне диференціальне рівняння зі стрибками вигляду
	$$\d X(t) = a(X(t))\d t + \sum_{k = 1}^m b_k(X(t-))\d Z_k(t),$$
	де $Z_k, \ k \in \{1, 2, ..., m\},$ -- незалежні центровані процеси Леві зі скінченним другим моментом.	Ми доведемо, що якщо коефіцієнт $a(t)$ має певну степеневу асимптотику при $t \to \infty$, а коефіцієнти $b_k, \ k \in \{1, 2, ..., m\},$ задовольняють певну умову на зростання, то розв'язок $X(t)$ м.~н. має таку саму асимптотику при $t \to \infty$, що і розв'язок звичайного диференціального рівняння (\ref{ode}).
	
	Основна частина статті має наступну структуру. У розділі 2 ми доводимо деякі леми про асимптотичну поведінку стохастичних інтегралів за вінерівським процесом та за компенсованою пуассонівською мірою. У розділі 3 ми доводимо два основні результати про асимптотичну поведінку розв'язків стохастичних диференціальних рівнянь зі стрибками: теорему \ref{theorem_asymptotics_constant}, в якій коефіцієнт зносу еквівалентний до деякої додатної сталої, та теорему \ref{theorem_asymptotics_power}, у якій коефіцієнт зносу еквівалентний до деякої додатної степеневої функції; в обох теоремах накладається деяка умова про швидкість зростання характеристик шуму. У розділ 4 ми винесли доведення деяких лем, необхідних для доведення теорем у розділі 3.
	
	\section{Асимптотика стохастичних інтегралів}
	
	У цьому параграфі ми отримаємо деякі допоміжні результати щодо асимптотики стохастичних інтегралів зі змінною верхньою межею $t$ при $t \to \infty$.
	
	Нехай $(\Omega, \F, \P)$ -- імовірнісний простір з потоком $\FF = (\F_t)_{t \geq 0}$, $W = W(t)$ -- $\FF$-вінерівський процес, $N = N(\d t, \d u)$ -- $\FF$-пуассонівська випадкова міра на\footnote{$\R_+$ позначає множину невід'ємних дійсних чисел.} $\R_+ \times \R$, незалежна від $W$, з характеристичною мірою $\d t \cdot \nu(\d u)$, де міра $\nu$ така, що $\int_{\R} u^2 \nu(\d u) < \infty$, $\tilde N = \tilde N(\d t, \d u) := N(\d t, \d u) - \d t \cdot \nu(\d u)$.
	
	\begin{lemma}\label{martingale_asymptotics}
		Нехай $M = M(t)$ -- квадратично інтегровний мартингал. Якщо $\E M^2(t) = O\left(t^\gamma\right), \ t \to \infty,$ для деякого $\gamma < 2$, то м.~н. $\frac{M(t)}{t} \to 0, \ t \to \infty$.
	\end{lemma}

	\begin{proof}
		З умови випливає, що існує $T \geq 0$ таке, що\footnote{Тут і надалі $C > 0$ -- універсальна константа, що може змінюватися від рядка до рядка.} $\E M^2(t) \leq Ct^\gamma, \ t \geq T$. Нехай $\varepsilon > 0$, $k \in \N$ таке, що $2^{k + 1} \geq T$. Оцінимо ймовірність:
		$$\P\left\{\sup_{2^k \leq t \leq 2^{k + 1}} \left|\frac{M(t)}t\right| \geq \varepsilon\right\} \leq \P\left\{\sup_{2^k \leq t \leq 2^{k + 1}} \frac{|M(t)|}{2^k} \geq \varepsilon\right\} \leq \P\left\{\sup_{t \leq 2^{k + 1}} |M(t)| \geq \varepsilon2^k\right\}$$
		\begin{equation*}\ltag{за нерівністю Дуба}
			\leq \frac{\E M^2(2^{k + 1})}{\left(\varepsilon 2^k\right)^2} \leq \frac{C\left(2^{k + 1}\right)^\gamma}{\varepsilon^2 2^{2k}} = \frac{C2^\gamma}\varepsilon \left(2^{\gamma - 2}\right)^k.
		\end{equation*}
		Для $n \in \N$ таких, що $2^{n + 1} \geq T$,
		$$\P\left\{\limsup_{t \to \infty}\left|\frac{M(t)}t\right| \geq \varepsilon\right\} \leq \P\left\{\sup_{t \geq 2^n} \left|\frac{M(t)}t\right| \geq \varepsilon\right\} \leq$$
		$$\leq \sum_{k = n}^\infty \P\left\{\sup_{2^k \leq t \leq 2^{k + 1}} \left|\frac{M(t)}t\right| \geq \varepsilon\right\} \leq \frac{C2^\gamma}\varepsilon \sum_{k = n}^\infty \left(2^{\gamma - 2}\right)^k.$$
		Останній ряд збігається до 0 при $n \to \infty$ (оскільки $\gamma - 2 < 0$), тому ймовірність на початку ланцюга нерівностей дорівнює 0. Оскільки $\varepsilon > 0$ довільне,
		$$\P\left\{\limsup_{t \to \infty}\left|\frac{M(t)}t\right| > 0\right\} = 0 \Longrightarrow \P\left\{\limsup_{t \to \infty}\left|\frac{M(t)}t\right| = 0\right\} = 1$$
		$$\Longrightarrow \P\left\{\lim_{t \to \infty}\left|\frac{M(t)}t\right| = 0\right\} = 1 \Longrightarrow \P\left\{\lim_{t \to \infty}\frac{M(t)}t = 0\right\} = 1,$$
		що й треба було довести.
	\end{proof}
	
	\begin{corollary}\label{ito_integral_asymptotics}
		Нехай прогресивно вимірний випадковий процес $b = b(t)$ такий, що $\E b^2(t) \leq C\left(1 + t^{2\beta}\right), \ t \geq 0,$ для деякого $0 \leq \beta < \frac12$. Тоді м.~н.
		$$\frac1t\int_0^t b(s) \d W(s) \to 0, \ t \to \infty.$$
	\end{corollary}

	\begin{proof}
		Покладемо $M(t) = \int_0^t b(s) \d W(s)$. Тоді
		\begin{equation*}\ltag{за ізометрією Іто}
			\E M^2(t) = \E \int_0^t b^2(s) \d s
		\end{equation*}
		\begin{equation*}\ltag{за теоремою Фубіні}
			= \int_0^t \E b^2(s) \d s \leq \int_0^t C(1 + s^{2\beta}) \d s = O(t^{2\beta + 1}), \ t \to \infty.
		\end{equation*}
		Застосовуючи лему \ref{martingale_asymptotics}, отримуємо результат наслідку.
	\end{proof}
	
	Позначимо через $\PP$ сигма-алгебру, породжену випадковими полями вигляду $c(t, u) = \zeta_0\I_{t = 0, u \in U_0} + \sum_{k = 1}^{n} \zeta_k\I_{t \in (t_{k - 1}, t_k], u \in U_k},$ де $n \in \N$, $\zeta_0$ є $\F_0$-вимірною випадковою величиною, $\zeta_k$ є $\F_{t_{k - 1}}$-вимірною випадковою величиною, $k \in \{1, 2, ..., n\}$, $U_k \in \B(\R)$, $k \in \{0, 1, 2, ..., n\}$, $0 = t_0 < t_1 < ... < t_n = \infty$.
	
	\begin{corollary}\label{jump_integral_asymptotics}
		Нехай $\PP$-вимірне випадкове поле $c = c(t, u)$ таке, що 
		$$\E \int_{\R} c^2(t, u) \nu(\d u) \leq C\left(1 + t^{2\beta}\right), \ t \geq 0,$$
		для деякого $0 \leq \beta < \frac12$. Тоді м.~н.
		$$\frac1t\int_0^t\int_{\R}c(s, u)\tilde N(\d s, \d u) \to 0, \ t \to \infty.$$
	\end{corollary}
	\begin{proof}
		Покладемо $M(t) = \int_0^t\int_{\R}c(s, u)\tilde N(\d s, \d u)$. Тоді
		\begin{equation*}\ltag{за ізометрією Іто}
			\E M^2(t) = \E \int_0^t\int_{\R} c^2(s, u) \nu(\d u) \d s
		\end{equation*}
		\begin{equation*}\ltag{за теоремою Фубіні}
			= \int_0^t \E\int_{\R} c^2(s, u) \nu(\d u) \d s
		\end{equation*}
		$$\leq \int_0^t C(1 + s^{2\beta}) \d s = O(t^{2\beta + 1}), \ t \to \infty.$$
		Застосовуючи лему \ref{martingale_asymptotics}, отримуємо результат наслідку.
	\end{proof}

	
	\section{Асимптотика розв'язків стохастичних рівнянь}
	
	Нехай $W_k$ -- вінерівський процес, $k \in \{1, 2, ..., m\}$, $\tilde N_k$ -- компенсована пуассонівська міра з компенсатором $\d t \cdot \nu_k(\d u)$, де міра $\nu_k$ така, що $\int_\R u^2\nu_k(\d u) < \infty$, $k \in \{1, 2, ..., l\}$, причому $W_1, W_2, ..., W_m, \tilde N_1, \tilde N_2, ..., \tilde N_l$ незалежні.
	
	Наступна теорема встановлює еквівалентність розв'язків стохастичних та звичайних диференціальних рівнянь у випадку, коли коефіцієнт зносу має додатну границю при $t \to \infty$, а характеристики шуму зростають не дуже швидко. Ця теорема є важливим результатом, який використовується далі (див. теорему \ref{theorem_asymptotics_power}) при встановленні степеневого типу зростання для розв'язків стохастичних диференціальних рівнянь, коефіцієнти яких мають степеневе зростання.
	
	\begin{theorem}\label{theorem_asymptotics_constant}
		Нехай $a = a(t)$ та $b_k = b_k(t), \ k \in \{1, 2, ..., m\}$ -- прогресивно вимірні\footnote{Випадковий процес $a = a(t)$ назвемо прогресивно вимірним, якщо для будь-яких $t \geq 0$ звуження відображення $a$ на множину $[0, t] \times \Omega$ є вимірним відносно сигма-алгебри $\B([0, t]) \otimes \F_t$.} випадкові процеси, $c_k = c_k(t, u), \ k \in \{1, 2, ..., l\}$ -- $\PP$-вимірне випадкове поле та нехай $\FF$-узгоджений c\`adl\`ag випадковий процес $X = X(t)$ має стохастичний диференціал
		$$\d X(t) = a(t)\d t + \sum_{k = 1}^m b_k(t)\d W_k(t) + \sum_{k = 1}^l \int_\R c_k(t, u)\tilde N_k(\d t, \d u),$$
		причому $\E X^2(0) < \infty$.
		Припустимо, що:
		\begin{enumerate}[label=(\Alph*)]
			\item випадковий процес $a$ обмежений та $a(t) \to A, \ t \to \infty,$ м.~н., де $A > 0$ -- випадкова величина;
			\item для деякого $0 \leq \beta < \frac12$
			\begin{equation}\label{diffusion_estimate}
				\sum_{k = 1}^m b_k^2(t) + \sum_{k = 1}^l \int_\R c_k^2(t, u) \nu_k(\d u) \leq C(1 + |X(t-)|^{2\beta}).
			\end{equation}
		\end{enumerate}
		Тоді $X(t) \sim At, \ t \to \infty,$ м.~н.
	\end{theorem}

	\begin{proof}
		Запишемо процес $X$ в інтегральній формі:
		$$X(t) = X(0) + \int_0^t a(s)\d s + \sum_{k = 1}^m\int_0^t b_k(s) \d W_k(s) + \sum_{k = 1}^l\int_0^t \int_\R c_k(s, u) \tilde N_k(\d s, \d u).$$
		
		\emph{Крок 1.} Спочатку перевіримо, що $\E X^2(t) \leq C(1 + t^2), \ t \geq 0,$ для деякого $C > 0$. Аналогічно до леми 3.3.2 з книги \cite{kunita2019stochastic} можна перевірити, що з умови (\ref{diffusion_estimate}) випливає, що $\sup_{0 \leq t \leq T}\E X^2(t) < \infty, \ T \geq 0$. За нерівністю Коші--Буняковського
		$$\frac14\E X^2(t) \leq \E X^2(0) + \E\left(\int_0^t a(s)\d s\right)^2$$
		$$+ \E\left[\left(\sum_{k = 1}^m\int_0^t b_k(s) \d W_k(s)\right)^2 + \left(\sum_{k = 1}^l\int_0^t \int_\R c_k(s, u) \tilde N_k(\d s, \d u)\right)^2\right]$$
		$$=: E_1 + E_2(t) + E_3(t).$$
		Оцінимо доданки у правій частині:
		$$E_1 = \E X^2(0) < \infty \qquad \text{за припущенням};$$
		$$E_2(t) = \E\left(\int_0^t a(s)\d s\right)^2 \leq Ct^2, \qquad \text{бо $a$ обмежене};$$
		$$E_3(t) = \E\left[\left(\sum_{k = 1}^m\int_0^t b_k(s) \d W_k(s)\right)^2 + \left(\sum_{k = 1}^l \int_0^t \int_\R c_k^2(s, u) \tilde N_k(\d s, \d u)\right)^2\right]$$
		\begin{equation*}\ltag{оскільки доданки сум незалежні}
			= \E\left[\sum_{k = 1}^m\left(\int_0^t b_k(s) \d W_k(s)\right)^2 + \sum_{k = 1}^l \int_0^t \left(\int_\R c_k^2(s, u) \tilde N_k(\d s, \d u)\right)^2\right]
		\end{equation*}
		\begin{equation*}\ltag{за ізометрією Іто}
			= \E\left[\sum_{k = 1}^m\int_0^t b_k^2(s) \d s + \sum_{k = 1}^l\int_0^t \int_\R c_k^2(s, u) \nu(\d u)\d s\right]
		\end{equation*}
		\begin{equation*}\ltag{за умовою (Б)}
			\leq C\E\int_0^t \left(1 + |X(s-)|^{2\beta}\right) \d s
		\end{equation*}
		\begin{equation*}\ltag{за теоремою Фубіні}
			= \int_0^t\E\left(C(1 + |X(s-)|^{2\beta})\d s\right)
		\end{equation*}
		\begin{equation*}\ltag{за нерівністю Єнсена}
			\leq C\left(t + \int_0^t\left(\E X^2(s-)\right)^\beta\d s\right).
		\end{equation*}
		Таким чином, отримали оцінку
		$$\E X^2(t) \leq C(1 + t^2) + C\int_0^t \left(\E X^2(s-)\right)^\beta \d s.$$
		Використовуючи лему Вендроффа (див. теорему 7.3 у \cite{mao1994exponential}), яка є узагальненням леми Гронуолла-Беллмана, отримуємо
		$$\E X^2(t) \leq C\left((1 - \beta)t + (1 + t^2)^{1 - \beta}\right)^\frac1{1 - \beta},$$
		з чого неважко вивести
		$$\E X^2(t) \leq C(1 + t^2), \ t \geq 0.$$
		
		\emph{Крок 2.} Тепер знайдемо асимптотику розв'язку $X(t)$ при $t \to \infty$. Поділимо стохастичне диференціальне рівняння на $t > 0$:
		$$\frac{X(t)}t = \frac{X(0)}{t} + \frac1t\int_0^t a(s)\d s$$
		$$+ \sum_{k = 1}^m \frac1t\int_0^t b_k(s)\d W_k(s) + \sum_{k = 1}^l \frac1t\int_0^t\int_\R c_k(s, u)\tilde N_k(\d s, \d u)$$
		$$=: T_1(t) + T_2(t) + T_3(t) + T_4(t).$$
		Дослідимо збіжність доданків у правій частині при $t \to \infty$. Маємо $T_1(t) = \frac{X(0)}t \to 0, \ t \to \infty$. З умов теореми випливає, що
		$$\lim_{t \to \infty} T_2(t) = \lim_{t \to \infty}\frac1t\int_0^t a(s)\d s = \lim_{t \to \infty} a(t) = A \ \text{м.~н.}$$
		Для оцінки доданка $T_3$ зазначимо, що
		$$\E b_k^2(t) \leq \E \left(C(1 + |X(t-)|^{2\beta})\right) = C\left(1 + \E \left(X^2(t-)\right)^\beta\right)$$
		\begin{equation*}\ltag{за нерівністю Єнсена}
			\leq C\left(1 + \left(\E X^2(t-)\right)^\beta\right)
		\end{equation*}
		$$\leq C\left(1 + \left(C(1 + t^2)\right)^\beta\right) \leq C(1 + t^{2\beta}), \ k \in \{1, 2, ..., m\},$$
		тому за наслідком \ref{ito_integral_asymptotics}
		$$T_3(t) = \sum_{k = 1}^m \frac1t\int_0^t b_k(s) \d W_k(s) \to 0, \ t \to \infty, \ \text{м.~н.}$$
		Аналогічно попередньому пункту,
		$$\E \int_{\R} c_k^2(t, u) \nu_k(\d u) \leq C(1 + t^{2\beta}), \ k \in \{1, 2, ..., l\},$$
		тому за наслідком \ref{jump_integral_asymptotics}
		$$T_4(t) = \sum_{k = 1}^l\frac1t\int_0^t\int_\R c_k(s, u)\tilde N_k(\d s, \d u) \to 0, \ t \to \infty, \ \text{м.~н.}$$
		Таким чином, з одержаних збіжностей отримуємо твердження теореми.
	\end{proof}

	Якщо замість умови (А) у попередній теоремі розглянути умову
	\begin{enumerate}
		\item[(A')] $A_- \leq a(t) \leq A_+, \ t \geq 0,$ де $A_- > 0, \ A_+ > 0$ -- випадкові величини,
	\end{enumerate}
	то можна довести наступний результат.
	
	\begin{theorem}
		Нехай виконуються умови (А') та (Б). Тоді м.~н.
		$$A_- \leq \liminf_{t \to \infty} \frac{X(t)}t \leq \limsup_{t \to \infty} \frac{X(t)}t \leq A_+.$$
	\end{theorem}
		
	Наступна теорема є головним результатом даної статті.
	
	\begin{theorem}\label{theorem_asymptotics_power}
		Нехай $X$ -- деякий (необов'язково єдиний) розв'язок стохастичного диференціального рівняння
		$$\d X(t) = a(X(t))\d t + \sum_{k = 1}^m b_k(X(t-))\d Z_k(t),$$
		де $a = a(x)$, $b_k = b_k(x), \ k \in \{1, 2, ..., m\}$ -- локально обмежені вимірні функції, $Z_k = Z_k(t), \ k \in \{1, 2, ..., m\}$ -- незалежні центровані процеси Леві зі скінченним другим моментом, $\E X^2(0) < \infty$. Нехай $\alpha \in [0, 1)$.
		Припустимо, що:
		\begin{enumerate}[label=(\Alph*)]
			\item $a(x) \sim Ax^\alpha, \ x \to +\infty,$ де $A > 0$ -- невипадкова стала;
			\item для деякого $2\beta \in [0, 1 + \alpha)$
			$$\sum_{k = 1}^m b_k^2(x) \leq C\left(1 + |x|^{2\beta}\right);$$
			\item $X(t) \to +\infty, \ t \to \infty, \ \text{м.~н.}$
		\end{enumerate}
		Тоді
		\begin{equation}\label{solution_asymptotics}
			X(t) \sim \left((1 - \alpha)At\right)^\frac1{1 - \alpha}, \ t \to \infty, \ \text{м.~н.}
		\end{equation}
	\end{theorem}

	\begin{remark}
		Умова (В) є суттєвою та не випливає з умов (А)-(Б). Вона зустрічається у працях \cite{buldygin2018asymptotic}, \cite{gikhman1982stochastic}. Для виконання умови (В) достатньо, наприклад, виконання наступних умов:
		\begin{itemize}
			\item коефіцієнти $a$ та $b_k, \ k \in \{1, 2, ..., m\},$ задовольняють умову Ліпшиця;
			\item $\lim_{|x| \to +\infty} \frac{a(x)}{|x|^\alpha} > 0;$
			\item для деякого $k \in \{1, 2, ..., m\}$ виконуються умови:
			\begin{itemize}
				\item[\ding{51}] $\inf_{|x| \leq R} |b_k(x)| > 0, \ R > 0,$
				\item[\ding{51}] $Z_k(t), \ t \geq 0,$ має невироджену гауссівську компоненту або додатні стрибки з імовірністю 1.
			\end{itemize}
		\end{itemize}
	\end{remark}

	\begin{proof}
		Справедливість теореми у випадку $\alpha = 0$ випливає з теореми \ref{theorem_asymptotics_constant}; надалі вважаємо, що $\alpha \in (0, 1)$. З умови випливає, що процеси $Z_k, \ k \in \{1, 2, ..., m\},$ допускають представлення
		$$\d Z_k(t) = \sigma_k\d W_k(t) + \int_\R u\tilde N_k(\d t, \d u), \ k \in \{1, 2, ..., m\}.$$
		Тут $\sigma_k \geq 0$, $W_k$ -- вінерівський процес, $\tilde N_k$ -- компенсована пуассонівська міра з компенсатором $\d t \cdot \nu_k(\d u)$, де міра $\nu_k$ така, що $\int_\R u^2\nu_k(\d u) < \infty$, $k \in \{1, 2, ..., m\}$. При цьому $W_1, W_2, ..., W_m, \tilde N_1, \tilde N_2, ..., \tilde N_m$ незалежні. Далі для скорочення позначень проведемо міркування лише для рівняння
		\begin{equation}\label{sde_with_jumps}
			\d X(t) = a(X(t))\d t + b(X(t))\d W(t) + \int_R c(X(t-))u\tilde N(\d t, \d u),
		\end{equation}
		де $b = b(x), c = c(x)$ -- локально обмежені вимірні функції, $W = W(t)$ -- вінерівський процес, $\tilde N = \tilde N(\d t, \d u)$ -- компенсована пуассонівська випадкова міра з компенсатором $\d t \cdot \nu(\d u)$, де міра $\nu$ така, що $\int_\R u^2 \nu(\d u) < \infty$, процеси та міри $W_1, W_2, ..., W_m, \tilde N_1, \tilde N_2, ..., \tilde N_m$ незалежні та
		$$b^2(x) + c^2(x) \leq C\left(1 + |x|^{2\beta}\right).$$
		
		Візьмемо таку двічі неперервно диференційовну функцію $f = f(x)$, що
		\begin{equation*}
			f(x) = \begin{cases}
				0  & x \leq 0 \\
				\frac{x^{1 - \alpha}}{1 - \alpha} & x \geq 1,
			\end{cases}
		\end{equation*}
		причому $f(x) \leq \frac{x^{1 - \alpha}}{1 - \alpha}, \ 0 < x < 1$. Позначимо $\tilde X(t) = f(X(t))$. За формулою Іто
		\begin{equation}\label{sde_with_jumps_new}
			\d \tilde X(t) = \tilde a(t)\d t + \tilde b(t) \d W(t) + \int_\R \tilde c(t, u)\tilde N(\d t, \d u),
		\end{equation}
		де
		$$\tilde a(t) = a(X(t-))f'(X(t-)) + \frac12 b^2(X(t-))f''(X(t-))$$
		$$+ \int_\R\Big(f\big(X(t-) + c(X(t-))u\big) - f(X(t-)) - c(X(t-))uf'(X(t-))\Big)\nu(\d u)$$
		$$=: \tilde a_1(t) + \tilde a_2(t) + \tilde a_3(t),$$
		$$\tilde b(t) = b(X(t-))f'(X(t-)),\qquad \tilde c(t, u) = f\big(X(t-) + c(X(t-))u\big) - f(X(t-)).$$
		
		Для подальшого доведення необхідні наступні дві леми, доведення яких ми винесли у додаток (розділ 4).
			
		\begin{lemma}\label{convergence_lemma}
			$$\int_\R\big(f\left(x + c(x)u\right) - f(x) - f'(x)c(x)u\big)\nu(\d u) \to 0, \ x \to +\infty.$$
		\end{lemma}
		
		\begin{lemma}\label{estimate_lemma}
			$$\int_\R\big(f\left(x + c(x)u\right) - f(x)\big)^2\nu(\d u) \leq Cx^{2(\beta - \alpha)}, \ x \geq 1.$$
		\end{lemma}
		
		Перевіримо, що нове стохастичне диференціальне рівняння (\ref{sde_with_jumps_new}) задовольняє умови теореми \ref{theorem_asymptotics_constant}.
		
		Помітимо, що коефіцієнт $\tilde a$ обмежений. Дослідимо його асимптотичну поведінку, дослідивши поведінку кожного з доданків $\tilde a_1(t), \tilde a_2(t), \tilde a_3(t)$:
		$$\lim_{t \to \infty} \tilde a_1(t) = \lim_{t \to \infty} a(X(t))f'(X(t)) = \lim_{t \to \infty} \frac{a(X(t))}{X^\alpha(t)} = \lim_{t \to \infty} \frac{AX^\alpha(t)}{X^\alpha(t)} = A \ \text{м.~н.};$$
		$$\lim_{t \to \infty}|\tilde a_2(t)| = \lim_{t \to \infty} \left|b^2(X(t))f''(X(t))\right| \leq \lim_{t \to \infty}\frac{C|\alpha|X^{2\beta}(t)}{X^{1 + \alpha}(t)}$$
		$$\leq C\lim_{t \to \infty}\frac 1{X^{1 + \alpha - 2\beta}(t)} = 0, \ t \to \infty, \ \text{м.~н.,}$$
		оскільки $1 + \alpha - 2\beta > 0$ та $X(t) \to \infty, \ t \to \infty,$ м.~н.;
		$$\lim_{t \to \infty} \tilde a_3(t) = \lim_{t \to \infty}\int_\R\Big(f\big(X(t-) + c(X(t-))u\big) - f(X(t-))$$
		$$- c(X(t-))uf'(X(t-))\Big)\nu(\d u) = 0, \ t \to \infty, \ \text{м.~н.}$$
		за лемою \ref{convergence_lemma}, оскільки $X(t) \to \infty, \ t \to \infty,$ м.~н.
		Отже, $\lim_{t \to \infty} \tilde a(t) = A$ м.~н.
		
		Оцінимо коефіцієнт $\tilde b$. Якщо $X(t-) \geq 1$, то
		$$\tilde b^2(t) = \big(b(X(t-))f'(X(t-))\big)^2 = b^2(X(t-))\left(f'(X(t-))\right)^2 \leq \frac{CX^{2\beta}(t-)}{X^{2\alpha}(t-)}$$
		$$\leq CX^{2(\beta - \alpha)}(t-) = C\tilde X^\frac{2(\beta - \alpha)}{1 - \alpha}(t-).$$
		Якщо ж $X(t-) < 1$, то $\tilde b^2(t)$ обмежене рівномірно за $t$ невипадковою сталою, оскільки $|f'(x)| \leq C$.
		Отже, $b^2(t) \leq C\left(1 + |\tilde X(t-)|^{2\tilde\beta}\right), \ t \geq 0,$ де $2\tilde\beta := \frac{2(\beta - \alpha)}{1 - \alpha} < 1.$
		
		Оцінимо коефіцієнт $\tilde c$. Якщо $X(t-) \geq 1$, то
		$$\int_\R \tilde c^2(t, u)\nu(\d u) = \int_\R\left(f\left(X(t-) + c(X(t-))u\right) - f(X(t-))\right)^2\nu(\d u)$$
		\begin{equation*}\ltag{за лемою \ref{estimate_lemma}}
			\leq CX^{2(\beta - \alpha)}(t-) = C\tilde X^\frac{2(\beta - \alpha)}{1 - \alpha}(t-).
		\end{equation*}
		Якщо ж $X(t-) < 1$, то за формулою Тейлора
		$$\Big(f\big(X(t-) + c(X(t-))u\big) - f(X(t-))\Big)^2 = \big(f'(\xi_{X(t-), u})c(X(t-))u\big)^2 \leq Cu^2,$$
		тому $\int_\R \tilde c^2(t, u)\nu(\d u) \leq C$ (тут $\xi$ таке саме, як у лемі \ref{estimate_lemma}). Отже,
		$$\int_\R \tilde c^2(t, u)\nu(\d u) \leq C\left(1 + |\tilde X(t-)|^{2\tilde\beta}\right).$$
		
		Таким чином, коефіцієнти $\tilde a, \tilde b, \tilde c$ нового стохастичного диференціального рівняння задовольняють умови теореми \ref{theorem_asymptotics_constant}. Отже, $\tilde X(t) \sim At, \ t \to \infty,$ м.~н. Виконуючи відповідну заміну, отримуємо еквівалентність (\ref{solution_asymptotics}), що й треба було довести.
	\end{proof}

	\section{Додаток}
	
		\begin{proof}[Доведення леми \ref{convergence_lemma}]
		Нехай $x \geq 1$. З умови випливає, що $c^2(x) \leq Cx^{2\beta}$. Розіб'ємо інтеграл:
		$$\int_\R[...]\nu(\d u) = \int_{|u| < Kx^{1 - \beta}}[...]\nu(\d u) + \int_{|u| \geq Kx^{1 - \beta}}[...]\nu(\d u),$$
		де $[...] := f\left(x + c(x)u\right) - f(x) - f'(x)c(x)u$, $K > 0$ -- деяка константа.
		
		Нехай спочатку $|u| < Kx^{1 - \beta}$. За формулою Тейлора
		$$f\left(x + c(x)u\right) - f(x) - f'(x)c(x)u = \frac12f''(\xi_{x, u})c^2(x)u^2,$$
		де $\xi_{x, u} \in \left[x \land \left(x + c(x)u\right), x \lor \left(x + c(x)u\right)\right]$. Маємо
		$$\left|\xi_{x, u} - x\right| \leq \left|c(x)u\right|$$
		$$\Longrightarrow \left(\xi_{x, u} - x\right)^2 \leq \left(c(x)u\right)^2 = c^2(x)u^2 \leq Cx^{2\beta}u^2$$
		$$\Longrightarrow \left|\xi_{x, u} - x\right| \leq Cx^\beta |u|.$$
		Оберемо $K$ так, що
		$$Cx^\beta |u| \leq Cx^\beta Kx^{1 - \beta} = CKx \leq \frac12x, \qquad x \geq 1, \ |u| < Kx^{1 - \beta},$$
		тому $\frac12x \leq \xi_{x, u} \leq \frac32x.$ Отже,
		$$\int_{|u| < Kx^{1 - \beta}}[...]\nu(\d u) = \int_{|u| < Kx^{1 - \beta}}f''(\xi_{x, u})c^2(x)u^2\nu(\d u)$$
		$$\leq Cx^{2\beta}\int_{|u| < Kx^{1 - \beta}}\frac{u^2}{\xi_{x, u}^{\alpha + 1}}\nu(\d u) \leq Cx^{2\beta} \left(\frac12x\right)^{-(\alpha + 1)}\int_\R u^2 \nu(\d u)$$
		$$\leq \frac C{x^{1 + \alpha - 2\beta}} \to 0, \ x \to +\infty,$$
		бо $1 + \alpha - 2\beta > 0$, а $\int_\R u^2 \nu(\d u) < \infty$. 
		
		Оцінимо тепер інтеграл за множиною $\{u \in \R: |u| \geq Kx^{1 - \beta}\}$. Розіб'ємо інтеграл
		$$\int_{|u| \geq Kx^{1 - \beta}}[...]\nu(\d u) = \int_{|u| \geq Kx^{1 - \beta}} f\left(x + c(x)u\right)\nu(\d u)$$
		$$- f(x)\int_{|u| \geq Kx^{1 - \beta}} \nu(\d u) - f'(x)c(x)\int_{|u| \geq Kx^{1 - \beta}} u\nu(\d u)$$
		$$=: I_1(x) - I_2(x) - I_3(x)$$
		та оцінимо кожен доданок у правій частині. Помітимо, що $f(x) \leq \frac{|x|^{1 - \alpha}}{1 - \alpha}, x \in \R.$
		$$I_1(x) = \int_{|u| \geq Kx^{1 - \beta}}f\left(x + c(x)u\right)\nu(\d u) \leq \frac1{1 - \alpha}\int_{|u| \geq Kx^{1 - \beta}}1 \cdot \left|x + c(x)u\right|^{1 - \alpha}\nu(\d u)$$
		\begin{equation*}\ltag{за нерівністю Гьольдера}
			\leq \frac1{1 - \alpha}\left(\int_{|u| \geq Kx^{1 - \beta}}1^\frac2{1 + \alpha}\nu(\d u)\right)^\frac{1 + \alpha}2\left(\int_{|u| \geq Kx^{1 - \beta}}\left(\left|x + c(x)u\right|^{1 - \alpha}\right)^\frac2{1 - \alpha}\nu(\d u)\right)^\frac{1 - \alpha}2
		\end{equation*}
		$$= \frac1{1 - \alpha}\left(\int_{|u| \geq Kx^{1 - \beta}}\nu(\d u)\right)^\frac{1 + \alpha}2\left(\int_{|u| \geq Kx^{1 - \beta}}\left(x + c(x)u\right)^2\nu(\d u)\right)^\frac{1 - \alpha}2.$$
		Окремо оцінимо кожен з інтегралів:
		$$\int_{|u| \geq Kx^{1 - \beta}}\nu(\d u) = \int_{u^2 \geq K^2x^{2 - 2\beta}}\nu(\d u)$$
		\begin{equation*}\ltag{за нерівністю Чебишова}
			~
		\end{equation*}
		\begin{equation}\label{estimate1}
			\leq \frac1{K^2x^{2 - 2\beta}}\int_\R u^2\nu(\d u) \leq \frac C{x^{2 - 2\beta}};
		\end{equation}
		$$\int_{|u| \geq Kx^{1 - \beta}}\left(x + c(x)u\right)^2\nu(\d u)$$
		\begin{equation*}\ltag{за нерівністю Коші--Буняковського}
			\leq 2x^2\int_{|u| \geq Kx^{1 - \beta}}\nu(\d u) + 2c^2(x)\int_{|u| \geq Kx^{1 - \beta}}u^2\nu(\d u)
		\end{equation*}
		\begin{equation*}\ltag{за оцінкою (\ref{estimate1})}
			\leq 2x^2 \frac C{x^{2 - 2\beta}} + 2c^2(x)\int_{|u| \geq Kx^{1 - \beta}}u^2\nu(\d u)
		\end{equation*}
		\begin{equation}\label{estimate2}
			\leq Cx^{2\beta} + Cx^{2\beta} \leq Cx^{2\beta}.
		\end{equation}
		Таким чином, за оцінками (\ref{estimate1}) та (\ref{estimate2})
		$$I_1(x) \leq \frac1{1 - \alpha}\left(\frac C{x^{2 - 2\beta}}\right)^\frac{1 + \alpha}2\left(Cx^{2\beta}\right)^\frac{1 - \alpha}2 \leq \frac C{x^{1 + \alpha - 2\beta}} \to 0, \ x \to +\infty,$$
		бо $1 + \alpha - 2\beta > 0$.
		$$I_2(x) = f(x)\int_{|u| \geq Kx^{1 - \beta}} \nu(\d u) \leq \frac{x^{1 - \alpha}}{1 - \alpha}\int_{|u| \geq Kx^{1 - \beta}}\nu(\d u)$$
		\begin{equation*}\ltag{за оцінкою (\ref{estimate1})}
			\leq \frac{x^{1 - \alpha}}{1 - \alpha}\frac C{x^{2 - 2\beta}} \leq \frac C{x^{1 + \alpha - 2\beta}} \to 0, \ x \to +\infty,
		\end{equation*}
		бо $1 + \alpha - 2\beta > 0$.
		$$|I_3(x)| = \left|f'(x)c(x)\int_{|u| \geq Kx^{1 - \beta}} u\nu(\d u)\right| \leq \frac 1{x^\alpha}Cx^\beta\int_{|u| \geq Kx^{1 - \beta}}|u|\nu(\d u)$$
		$$= \frac C{x^{\alpha - \beta}}\int_{|u| \geq Kx^{1 - \beta}}|u|\nu(\d u).$$
		Окремо оцінимо інтеграл:
		$$\int_{|u| \geq Kx^{1 - \beta}}|u|\nu(\d u) = \int_{|u| \geq Kx^{1 - \beta}}1 \cdot |u|\nu(\d u).$$
		\begin{equation*}\ltag{за нерівністю Коші--Буняковського}
			\left(\int_{|u| \geq Kx^{1 - \beta}}1^2\nu(\d u)\right)^\frac12\left(\int_{|u| \geq Kx^{1 - \beta}}|u|^2\nu(\d u)\right)^\frac12
		\end{equation*}
		$$\leq \left(\int_{|u| \geq Kx^{1 - \beta}}\nu(\d u)\right)^\frac12\left(\int_\R u^2\nu(\d u)\right)^\frac12 \leq C\left(\int_{|u| \geq Kx^{1 - \beta}}\nu(\d u)\right)^\frac12$$
		\begin{equation*}\ltag{за оцінкою (\ref{estimate1})}
			\leq C\left(\frac C{x^{2 - 2\beta}}\right)^\frac12 \leq \frac C{x^{1 - \beta}}.
		\end{equation*}
		Отже,
		$$|I_3(x)| \leq \frac C{x^{\alpha - \beta}} \frac C{x^{1 - \beta}} \leq \frac C{x^{1 + \alpha - 2\beta}} \to 0, \ x \to +\infty,$$
		бо $1 + \alpha -2\beta > 0$.
		
		Усі три доданки $I_1(x), I_2(x), I_3(x)$ прямують до нуля при $x \to +\infty$, тому
		$$\int_{|u| \geq Kx^{1 - \beta}}[...]\nu(\d u) \to 0, \ x \to +\infty.$$
		Таким чином, лема доведена.
	\end{proof}

	\begin{proof}[Доведення леми \ref{estimate_lemma}]
		Нехай $x \geq 1$. З умови випливає, що $c^2(x) \leq Cx^{2\beta}, \ x \geq 1.$
		Розіб'ємо інтеграл:
		$$\int_\R[...]\nu(\d u) = \int_{|u| < Kx^{1 - \beta}}[...]\nu(\d u) + \int_{|u| \geq Kx^{1 - \beta}}[...]\nu(\d u),$$
		де $[...] := \left(f\left(x + c(x)u\right) - f(x)\right)^2$, $K > 0$ -- деяка константа.
		
		Нехай спочатку $|u| < Kx^{1 - \beta}$. За формулою Лагранжа
		$$f\left(x + c(x)u\right) - f(x) = f'(\xi_{x, u})c(x)u,$$
		де $\xi_{x, u} \in \left[x \land \left(x + c(x)u\right), x \lor \left(x + c(x)u\right)\right]$. Як і в доведенні леми \ref{convergence_lemma}, маємо $\frac12x \leq \xi_{x, u} \leq \frac32x,$ якщо обрати достатньо мале $K$. Тому
		$$\int_{|u| < Kx^{1 - \beta}}[...]\nu(\d u) = \int_{|u| < Kx^{1 - \beta}}\left(f'(\xi_{x, u})\right)^2c^2(x)u^2\nu(\d u)$$
		$$\leq Cx^{2\beta}\int_{|u| < Kx^{1 - \beta}}\frac{u^2}{\xi_{x, u}^{2\alpha}}\nu(\d u).$$
		Оскільки $\alpha > 0$, то
		$$\int_{|u| < Kx^{1 - \beta}}\frac{u^2}{\xi_{x, u}^{2\alpha}}\nu(\d u) \leq \left(\frac12 x\right)^{-2\alpha}\int_\R u^2 \nu(\d u) \leq Cx^{-2\alpha}.$$
		Таким чином,
		$$\int_{|u| < Kx^{1 - \beta}}[...]\nu(\d u) \leq Cx^{2\beta}Cx^{-2\alpha} \leq Cx^{2(\beta - \alpha)}.$$
		
		Нехай тепер $|u| \geq Kx^{1 - \beta}$. 
		Оцінимо інтеграл:
		$$\frac12\int_{|u| \geq Kx^{1 - \beta}}[...]\nu(\d u)$$
		\begin{equation}\ltag{за нерівністю Коші--Буняковського}
			\leq \int_{|u| \geq Kx^{1 - \beta}} f^2\left(x + c(x)u\right)\nu(\d u) + f^2(x)\int_{|u| \geq Kx^{1 - \beta}}\nu(\d u)
		\end{equation}
		$$=: J_1(x) + J_2(x).$$
		Оцінимо кожен доданок у правій частині, використовуючи оцінки, отримані при доведенні леми \ref{convergence_lemma}:
		
		$$J_1(x) = \int_{|u| \geq Kx^{1 - \beta}} f^2\left(x + c(x)u\right)\nu(\d u)$$
		$$\leq \frac1{(1 - \alpha)^2}\int_{|u| \geq Kx^{1 - \beta}} 1 \cdot |x + c(x)u|^{2(1 - \alpha)}\nu(\d u)$$
		\begin{equation*}\ltag{за нерівністю Гьольдера}
			\leq \frac1{(1 - \alpha)^2} \left(\int_{|u| \geq Kx^{1 - \beta}}1^\frac1\alpha\nu(\d u)\right)^\alpha\left(\int_{|u| \geq Kx^{1 - \beta}}\left(|x + c(x)u|^{2(1 - \alpha)}\right)^\frac1{1 - \alpha}\nu(\d u)\right)^{1 - \alpha}
		\end{equation*}
		$$= \frac1{(1 - \alpha)^2} \left(\int_{|u| \geq Kx^{1 - \beta}}\nu(\d u)\right)^\alpha\left(\int_{|u| \geq Kx^{1 - \beta}}\left(x + c(x)u\right)^2\nu(\d u)\right)^{1 - \alpha}$$
		\begin{equation*}\ltag{за оцінками (\ref{estimate1}) та (\ref{estimate2})}
			\leq \frac1{(1 - \alpha)^2} \left(\frac C{x^{2 - 2\beta}}\right)^\alpha\left(Cx^{2\beta}\right)^{1 - \alpha} \leq Cx^{2(\beta - \alpha)};
		\end{equation*}
		$$J_2(x) = f^2(x)\int_{|u| \geq Kx^{1 - \beta}} \nu(\d u) \leq \left(\frac{x^{1 - \alpha}}{1 - \alpha}\right)^2 \int_{|u| \geq Kx^{1 - \beta}} \nu(\d u)$$
		\begin{equation*}\ltag{за оцінкою (\ref{estimate1})}
			\leq \left(\frac{x^{1 - \alpha}}{1 - \alpha}\right)^2 \frac C{x^{2 - 2\beta}} \leq Cx^{2(\beta - \alpha)}.
		\end{equation*}
		
		Обидва доданки оцінюються як $Cx^{2(\beta - \alpha)}$, тому
		$$\int_{|u| \geq Kx^{1 - \beta}}[...]\nu(\d u) \leq Cx^{2(\beta - \alpha)}.$$
		Таким чином, лема доведена.
	\end{proof}
	
	\bibliographystyle{plainnat}
	\bibliography{article}
	
\end{document}